\documentclass[11pt,reqno]{amsart}
\usepackage[T1]{fontenc}
\usepackage{lmodern,microtype,amsmath,amssymb,amsthm,mathtools,booktabs,longtable,array}
\usepackage[margin=1in]{geometry}
\usepackage{graphicx}
\usepackage{float}
\usepackage[hidelinks]{hyperref}
\hypersetup{pdftitle={Spectral extremal 1-planar graphs with bounded pentagon packing},pdfauthor={Zhanhe Zhang},pdfsubject={Spectral extremal graph theory, 1-planar graphs, and vertex-disjoint cycles}}
\newtheorem{theorem}{Theorem}[section]
\newtheorem{lemma}[theorem]{Lemma}
\newtheorem{proposition}[theorem]{Proposition}
\newtheorem{corollary}[theorem]{Corollary}
\theoremstyle{remark}\newtheorem{remark}[theorem]{Remark}
\newcommand{\one}{\mathbf 1}
\newcommand{\nf}{\nu_5}
\newcommand{\SPEX}{\operatorname{SPEX}_{\mathcal P_1}}
\title[Spectral extremal 1-planar graphs]{Spectral extremal 1-planar graphs with bounded pentagon packing}
\author{Zhanhe Zhang}
\address{School of International Trade and Economics, Central University of Finance and Economics, Beijing 102206, China}
\email{2024311337@email.cufe.edu.cn}
\date{}
\begin{document}
\raggedbottom
\begin{abstract}
Let $\mathcal P_1$ denote the class of 1-planar graphs and let $tC_5$ be the disjoint union of $t$ copies of $C_5$. For every fixed $t\ge3$ and all sufficiently large $n$, we determine the unique $n$-vertex $tC_5$-free graph in $\mathcal P_1$ with maximum adjacency spectral radius, answering Problem~1 of Li, Wang and Zhao. The proof first gives a structural description of every extremizer. After two dominating vertices are removed, the remainder consists of copies of the seven-vertex graph $B=K_1\vee2K_3$ together with at most one bounded connected core. This follows from a two-family covering theorem for 1-planar $K_2$-joins and the sharp packing--defect inequality
\[
 12v(Q)-7e(Q)\ge1-10\nu_5(Q).
\]
The same inequality yields an exact edge-extremal result for remainders with bounded pentagon packing. A normalized resolvent then cancels the repeated $B$-components, and finite moment comparisons force all packing and all nonzero defect into one core and identify that core uniquely in each residue class modulo $7$. The case $t=3$ is the pentagon-free boundary case and is completed by one exact finite component lemma.
\end{abstract}
\maketitle

\section{Introduction}\label{sec:intro}
A graph is \emph{1-planar} if it can be drawn in the plane so that every edge is crossed at most once. The class was introduced by Ringel~\cite{Ringel} and lies immediately beyond planar graphs in the hierarchy of sparse topological graph classes. Spectral extremal questions for planar graphs go back to the conjecture of Boots and Royle~\cite{BR} and, independently, Cao and Vince~\cite{CV}; Tait and Tobin~\cite{TT} proved that, for sufficiently large order, the unique planar graph of maximum adjacency spectral radius is $K_2\vee P_{n-2}$. Spectral extremal problems for 1-planar graphs were subsequently studied by Zhang, Wang and Wang~\cite{ZWW}, and more recent work has considered spectral Tur\'an problems inside the 1-planar class~\cite{XC,LWZ}.

Vertex-disjoint cycles form a second line of extremal problems, beginning with the classical packing theory of Erd\H{o}s and P\'osa~\cite{EP}. For planar graphs, exact and asymptotic Tur\'an questions for short cycles and disjoint cycles have been studied, among others, by Dowden~\cite{Dowden}, Cranston, Lidick\'y, Liu and Shantanam~\cite{CLLS}, and Li~\cite{Li}. Fang, Lin and Shi~\cite{FLS} combined the spectral and packing viewpoints: using the theorem of Tait and Tobin, they observed that for planar graphs the spectral problem forbidding $t$ vertex-disjoint cycles stabilizes for $t\ge3$, because the unrestricted extremizer $K_2\vee P_{n-2}$ already remains extremal under those constraints.

The 1-planar problem behaves differently. Li, Wang and Zhao~\cite{LWZ} determined the unique spectral extremizer when one forbids a single $C_5$ or two vertex-disjoint copies of $C_5$, and asked for the fixed-$t$ problem for every $t\ge3$~\cite[Problem~1]{LWZ}. Once their two-pole reduction is available, the remainder is no longer forced to be pentagon-free. In the fixed-$t$ setting considered here, Proposition~\ref{prop:reduction} shows that it may carry up to $t-3$ vertex-disjoint pentagons. These packing units can, a priori, be distributed among several connected components. The fixed-$t$ problem therefore requires determining how this packing is arranged across the connected components of an extremal remainder.

We determine the extremal graph for every fixed $t\ge3$ and all sufficiently large $n$. The first result describes the structure of the remainder. Put
\[
 B:=K_1\vee2K_3,
 \qquad
 \delta(Q):=12v(Q)-7e(Q).
\]
The graph $B$ has seven vertices and twelve edges, so $\delta(B)=0$. We prove that it is the unique connected pentagon-free remainder with a 1-planar $K_2$-join and zero defect. More generally, if $p=\nu_5(Q)\ge1$, then every connected remainder with a 1-planar $K_2$-join satisfies the sharp inequality
\[
 \delta(Q)\ge1-10p.
\]
The zero-defect block can therefore repeat without using the bounded defect budget, whereas every component carrying pentagon packing changes that budget.

For $t\ge4$ write $k=t-3$. Set
\[
 N_0(t):=\max\{10^{19},250t,[10^7(t+11)]^2\}.
\]
For $N=n-2$, let $E\in\{0,\ldots,6\}$ be the unique integer for which
\[
 D^*:=1-10k+E\equiv5N\pmod7.
\]
We use $\SPEX(n,F)$ for the set of $n$-vertex $F$-free 1-planar graphs of maximum adjacency spectral radius.

\begin{theorem}\label{thm:structure}
Fix $t\ge4$ and put $k=t-3$. For every $n\ge N_0(t)$, every graph $G\in\SPEX(n,tC_5)$ has the form
\[
 G=K_2\vee(qB\cup Q)
\]
for some $q\ge0$ and a single connected graph $Q$ satisfying
\[
 \nu_5(Q)=k,\qquad \delta(Q)=D^*=1-10k+E,
\]
and
\[
 v(Q)\le5k+10+3E\le5k+28.
\]
In particular, $Q$ is the only non-$B$ component of the remainder.
\end{theorem}

Thus all but $O_t(1)$ vertices lie in copies of one seven-vertex block, while all admissible pentagon packing is localized in one bounded component. After the minimum defect is fixed, normalized spectral moments strictly penalize splitting positive packing between components or leaving additional pentagon-free defect outside the packing component.

The bounded core can then be identified explicitly. On vertices $q_1,\ldots,q_m$, let $P_m^2$ be the square of the path and put
\[
 A_m=P_m^2+q_1q_4+q_{m-3}q_m,
 \qquad
 R_m^{(1)}=P_m^2+q_1q_4-q_3q_5,
\]
\[
 R_m^{(2)}=A_m-q_3q_5-q_{m-4}q_{m-2}.
\]
Set $M=5k+10=5t-5$ and define
\begin{center}\small
\begin{tabular}{ccl}
\toprule
$E$ & order & $C_{k,E}$ \\
\midrule
0 & $M$   & $R_M^{(2)}$ \\
1 & $M+3$ & $R_{M+3}^{(2)}-q_6q_8$ \\
2 & $M-1$ & $R_{M-1}^{(2)}$ \\
3 & $M+2$ & $R_{M+2}^{(2)}-q_5q_7$ \\
4 & $M-2$ & $R_{M-2}^{(2)}$ \\
5 & $M+1$ & $R_{M+1}^{(2)}-q_5q_7$ \\
6 & $M-3$ & $R_{M-3}^{(1)}$ \\
\bottomrule
\end{tabular}
\end{center}

\begin{theorem}\label{thm:main}
Under the hypotheses and notation of Theorem~\ref{thm:structure}, let $m_E=v(C_{k,E})$ and $q=(N-m_E)/7$. Then
\[
 \SPEX(n,tC_5)=\bigl\{K_2\vee(qB\cup C_{k,E})\bigr\}.
\]
\end{theorem}

The packing--defect inequality also gives a purely extremal consequence, independent of the spectral comparison.

\begin{corollary}\label{cor:edge}
For $k\ge1$ and $N\ge5k+13$,
\[
 \max_{\substack{v(H)=N,\ K_2\vee H\ \mathrm{1\!\!-planar}\\ \nu_5(H)\le k}} e(H)
 =\left\lfloor\frac{12N+10k-1}{7}\right\rfloor.
\]
For $k=0$ and $N\ge13$, the corresponding maximum is $\lfloor12N/7\rfloor$.
\end{corollary}

The case $t=3$ is a boundary case because the remainder is pentagon-free. Put
\[
 Q_8=R_8^{(2)},\quad Q_9=R_9^{(2)},\quad Q_{10}=R_{10}^{(2)},
\]
\[
 Q_{11}=R_{11}^{(2)}-q_5q_7,\quad
 Q_{12}=R_{12}^{(2)}-q_5q_7,\quad
 Q_{13}=R_{13}^{(2)}-q_6q_8.
\]
\begin{theorem}\label{thm:t3main}
Let $n\ge10^{19}$ and write $N=n-2=7q+r$, $0\le r\le6$. Then the unique graph in $\SPEX(n,3C_5)$ is $K_2\vee H_N$, where
\[
 H_N=
 \begin{cases}
 qB,&r=0,\\
 (q-1)B\cup Q_{7+r},&1\le r\le6.
 \end{cases}
\]
\end{theorem}

The argument starts from the two-pole reduction in Proposition~\ref{prop:reduction}, where the additional pole-edge and packing bounds are established. Theorem~\ref{thm:cover} confines each connected remainder with a 1-planar $K_2$-join to one of two path-square envelopes, and Theorem~\ref{thm:defect} turns this restriction into the packing--defect bound used in Corollary~\ref{cor:edge} and the finite-core reduction. The normalized resolvent removes the repeated $B$-components, after which moment comparisons identify the core. The first strict distinction occurs by the fifth walk moment.

Lemma~\ref{lem:finitecomponents} is obtained by exact finite enumeration after the structural reductions reduce the search to finitely many cases. Three exhaustive implementations and an independent cycle enumerator give matching exact certificates. The remaining structural and positive-packing arguments are uniform in the parameters.

\section{Preliminaries and the two-pole reduction}\label{sec:reduction}
Graphs are finite and simple; cycles are not required to be induced. Uniqueness is always understood up to graph isomorphism. Write $v(Q),e(Q)$ for order and size, and $\nf(Q)$ for the maximum number of pairwise vertex-disjoint pentagons. A 1-drawing is a drawing on the sphere with at most one crossing per edge, with no adjacent edges crossing. The poles of $K_2\vee H$ are denoted $z,w$.
Put
\[
 B=K_1\vee 2K_3,\qquad v(B)=7,\quad e(B)=12,
 \qquad \delta(Q)=12v(Q)-7e(Q).
\]
\begin{proposition}[Two-pole reduction]\label{prop:reduction}
Let $t\ge3$, let $n\ge\max\{10^{19},250t\}$, and let $G\in\SPEX(n,tC_5)$. Then
\[
 G=K_2\vee H,\qquad \Delta(H)\le6,\qquad \nu_5(H)\le t-3.
\]
If $\rho=\rho(G)$ and the Perron vector is normalized by $x_z=x_w=1$ at the two vertices of the $K_2$, then
\[
 \rho>\sqrt{2n},\qquad
 \frac2\rho\le x_u\le\frac2{\rho-6}\quad(u\in H).
\]
\end{proposition}
\begin{proof}
An extremizer is connected: otherwise joining a component of largest spectral radius to another component by a bridge preserves 1-planarity and the cycle-packing constraint and increases the spectral radius. We verify explicitly the hypotheses of \cite[Theorem 1.4]{LWZ} for $F=tC_5$. Here $P_m^{2+}$ denotes the graph obtained from the path $u_1u_2\cdots u_m$ by adding $u_1u_m$ and all edges $u_i u_{i+2}$ for $1\le i\le m-2$, as in~\cite{LWZ}. First, every vertex of $F$ has degree $2$, so the minimum-degree condition holds. Second, since $n-2\ge5t$, choose $t$ disjoint blocks of five consecutive vertices in $P_{n-2}^2$; each block contains a pentagon, so
\[
 tC_5\subseteq P_{n-2}^2\subseteq K_2\vee P_{n-2}^{2+}.
\]
Third, $K_2\vee I_{n-2}$ contains no pentagon, hence $tC_5\not\subseteq K_2\vee I_{n-2}$. Finally, $|V(F)|=5t$ and our assumption $n\ge250t$ implies $n\ge2|V(F)|$. Together with $n\ge10^{19}$, all assumptions of \cite[Theorem 1.4]{LWZ} are therefore satisfied. That theorem yields vertices $z,w$ with
\[
 N_G(z)\cap N_G(w)=V(G)\setminus\{z,w\},\qquad x_z=x_w=1,
\]
together with $\rho>\sqrt{2n}$ and the displayed Perron-coordinate bounds. Theorem~1.4 of~\cite{LWZ} gives these conclusions. The edge $zw$ is established below. Since every 1-planar graph is $K_{3,7}$-free \cite[Lemma 2.4]{LWZ}, the remainder has maximum degree at most six.

We next prove that the pole edge $zw$ is present and that $\nu_5(H)\le t-3$. If $zw$ is absent, choose $u$ in the remainder, delete its at most six internal edges and the at most two edges crossing $zu$ or $uw$, and draw $zw$ along a cleared thin neighbourhood of $zuw$. If a new packing used the pentagon $zwabcz$, replace it on the same five vertices by $zabwcz$ in the original graph. Thus no new forbidden packing occurs. The gain in the Perron edge-weight sum is at least
\[
 1-\frac{24}{(\rho-6)^2}-\frac4{\rho-6}>0,
\]
a contradiction. Hence the poles are adjacent.

Suppose that $H$ contains $t-2$ disjoint pentagons, and remove their $5(t-2)$ vertices. The remaining graph $R$ is $2P_4$-free, since two disjoint $P_4$'s can be completed using different poles. If $R$ contains a $P_4$, deleting its four vertices leaves a $P_4$-free graph, whose connected components are stars or triangles. At most $24$ edges meet the removed vertices. Thus
\[
 e(H)\le N+25t-30,\qquad N=n-2.
\]
On the same vertex set take $H_0$ to be $\lfloor N/4\rfloor$ copies of $K_4-e$ and isolated vertices. Its $K_2$-join is 1-planar with an uncrossed pole edge by the strip and gluing construction of Section~\ref{sec:constructions}. Since $H_0$ is pentagon-free, every pentagon in the join uses a pole, so the join is $tC_5$-free. Put $C=25t-30$. The assumption on $n$ gives $N\ge9C+40$ and
\[
 e(H_0)\ge\tfrac54N-5\ge\tfrac98(N+C)\ge\tfrac98e(H).
\]
Since $\rho>200$, $(\rho/(\rho-6))^2<9/8$. Consequently the difference in Perron edge-weight sums is at least
\[
 \frac{4e(H_0)}{\rho^2}-\frac{4e(H)}{(\rho-6)^2}>0.
\]
Common edges cancel in the actual difference, so the estimate remains valid when the edge sets intersect. This contradiction proves $\nu_5(H)\le t-3$.
\end{proof}

\section{Two-family covering}\label{sec:cover}
On vertices $1,\ldots,m$, let $P_m^2$ contain the pairs at distance one or two. Define
\begin{equation}\label{eq:parents}
 W_m=P_m^2+1m,\qquad
 A_m=P_m^2+14+(m-3)m\qquad(m\ge7).
\end{equation}
The sizes are $2m-2$ and $2m-1$, respectively. We use $A_m$ as an upper envelope; admissibility of the extremal subgraphs is established in Section~\ref{sec:constructions}. Its definition includes both $35$ and $(m-4)(m-2)$.

\begin{theorem}[Two-family covering]\label{thm:cover}
If $Q$ is connected, $v(Q)=m\ge7$, and $K_2\vee Q$ is 1-planar, then, after relabelling, $Q$ is a spanning subgraph of $W_m$ or $A_m$.
\end{theorem}
\begin{proof}
Extend $K_2\vee Q$ to an edge-maximal 1-planar graph $G'$ on the same vertices. It remains a join $K_2\vee Q'$, with $Q'$ connected, so $G'$ is 3-connected. Choose a rich drawing, using \cite[Lemma 4.2]{GHHMTZ}. In such a drawing, the four uncrossed boundary edges of every crossing kite belong to the true-planar skeleton; by \cite[Lemma 4.4]{GHHMTZ}, the kite occupies a quadrilateral face of that skeleton.

Call a crossing of the form $za\times wb$ a spoke--spoke crossing. Its boundary kite has cyclic order $z,w,a,b$, so $zw$ is an uncrossed boundary edge of its quadrilateral face. There are at most two such faces, one on either side of $zw$. Distinct such crossings have disjoint pairs of non-pole endpoints: using an endpoint twice would either cross the same spoke twice or cross one of the boundary spokes required to be uncrossed by richness.

First suppose that there is no spoke--spoke crossing. The drawing restricted to $zw$ and all pole spokes is a plane $K_2\vee I_m$. Label its non-pole vertices in strip order $q_1,\ldots,q_m$. Its faces are the two end triangles $zwq_1,zwq_m$ and the quadrilaterals $zq_iwq_{i+1}$ between consecutive vertices. An internal edge can cross this plane base at most once. Unless it crosses $zw$, its endpoints have strip indices differing by at most two. If it crosses $zw$, its endpoints must be $q_1,q_m$. There can be at most one such edge. Hence $Q'\subseteq W_m$.

Now suppose there are $s\in\{1,2\}$ spoke--spoke crossings. From each crossing $za\times wb$ remove $a$, retaining $b$. The removed vertices are distinct and none is a retained endpoint of another such crossing. Since $zw$ belongs to a kite boundary, it is globally uncrossed. The remaining pole spokes and $zw$ therefore form a plane base $\mathcal B\cong K_2\vee I_r$, where $r=m-s\ge5$.

\emph{The empty end face.} Mark the crossing point by $x$. The kite boundary is the uncrossed quadrilateral $zwab$, and its interior contains just the two crossing diagonals. In particular, the regions with boundaries $zwx$ and $zbx$ contain no vertex or edge of the retained base. Deleting $a$ and its incident edges removes the segment $zx$ of $za$ as well as its continuation to $a$. Those two regions merge into a region bounded by $zw$, the entire now-uncrossed arc $wb$, and $bz$; the former crossing point is simply a point of the arc $wb$. This is a facial triangle of $\mathcal B$, empty with respect to its retained vertices and edges. Thus $b$ is an end vertex of the strip order. The two kites, if present, are on opposite sides of $zw$, and yield the two different end triangles.

Label the strip vertices $u_1,\ldots,u_r$, with a retained endpoint at $u_1$ and, when $s=2$, the other at $u_r$. The base faces are
\[
 T_0=zwu_1,\qquad F_i=zu_iwu_{i+1}\ (1\le i<r),\qquad T_r=zwu_r.
\]
After parallel dual edges are identified, the face adjacency graph is the cycle
\[
 T_0,F_1,F_2,\ldots,F_{r-1},T_r,T_0;
\]
the final edge is dual to $zw$.

\emph{The location and neighbours of a removed vertex.} Retain $a$ only as a marked point. At the left end, the segment $zx$ of its former spoke lies in $T_0$, and $za$ crosses $wu_1$ at $x$. The segment $xa$ cannot cross a second base edge. It follows that $a$ lies in $F_1$, the face on the other side of $wu_1$; see Figure~\ref{fig:endface}. An original edge from $a$ to a retained non-pole vertex crosses at most one base edge. The faces reachable from $F_1$ in at most one step are $T_0,F_1,F_2$, whose non-pole boundary vertices are precisely $u_1,u_2,u_3$. Hence these are the only possible retained neighbours. At the right end the analogous set is $\{u_{r-2},u_{r-1},u_r\}$.

\emph{No edge between the two removed vertices.} When $s=2$, the two marked points lie in $F_1$ and $F_{r-1}$. Their distance in the simplified dual is
\[
 \operatorname{dist}(F_1,F_{r-1})=\min\{r-2,3\}\ge3\qquad(r\ge5).
\]
A curve joining them and avoiding base vertices must therefore cross at least three base edges, even if crossing $zw$ were allowed. An edge of the 1-drawing cannot do so. The two removed vertices are consequently nonadjacent.

Finally, an edge between retained vertices cannot cross $zw$ and has strip-index difference at most two. Reinsert each removed vertex at its corresponding end of the linear order. Its possible neighbours are the first three, or the last three, retained vertices. The retained edges and these possible new edges all lie in $P_m^2+14+(m-3)m$. Thus $Q'\subseteq A_m$, and restriction to $Q$ completes the proof.
\end{proof}

\begin{figure}[tb]
\centering
\includegraphics[width=.92\textwidth]{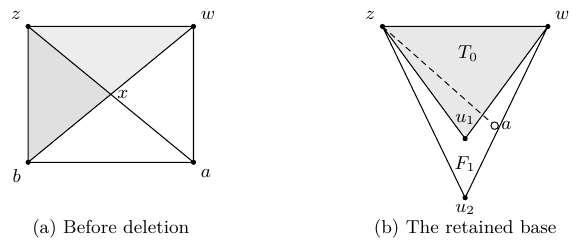}
\caption{The end-face argument. In (a), deleting $a$ merges the shaded regions into the empty triangle $zwb$. Panel (b) shows $T_0=zwu_1$ and $F_1=zu_1wu_2$, with $b=u_1$; other faces are omitted. The dashed former spoke and open point $a$ are deleted data, not part of the retained base.}\label{fig:endface}
\end{figure}

\begin{remark}
The wrapped family is necessary: for example, $P_8^2+18$ is a possible wrapped remainder. Throughout, $W_m$ and $A_m$ serve as upper envelopes; Section~\ref{sec:constructions} establishes admissibility of the extremal constructions.
\end{remark}

\section{Packing and defect}\label{sec:defect}
Czap--Hud\'ak--Madaras \cite[Lemma 3.5]{CHM} prove that a graph with a 1-planar $K_2$-join contains no $K_{3,1,1}$ as a subgraph. Equivalently, no edge has three common neighbours. We call this excluded configuration a three-page book. By \cite[Corollary 3.1]{CHM}, $\Delta(Q)\le6$.

\begin{lemma}[Canonical windows]\label{lem:windows}
For $m\ge5$, the unique pentagon of $P_m^2$ on positions $i,\ldots,i+4$ is
\[
 i\ (i+2)\ (i+4)\ (i+3)\ (i+1)\ i.
\]
There are $m-4$ such windows. A distance-two edge occurs in at most three of them, a distance-one edge in at most two, and either extra end chord of $A_m$ in none. If a surviving family has packing number at most $p$, it contains at most $5p$ windows.
\end{lemma}
\begin{proof}
The shorter arc between the least and greatest indexed vertices of a pentagon has at most two edges, so its index span is at most four. Five distinct indices force five consecutive positions, and the endpoint degrees force the displayed cycle. Edge incidences follow directly from the formula. A family of intervals is pierced by as many points as its maximum disjoint subfamily: greedily choose the right endpoint of the interval ending first. Each selected point belongs to at most five length-five integer intervals. Apply this to the surviving canonical windows, whose disjoint subfamilies are also pentagon packings in the graph.
\end{proof}

\begin{theorem}[Packing--defect bounds]\label{thm:defect}
Let $Q$ be connected and suppose $K_2\vee Q$ is 1-planar. Put $m=v(Q)$ and $p=\nf(Q)$.
\begin{enumerate}
\item If $p=0$, then $\delta(Q)\ge0$, with equality if and only if $Q\cong B$.
\item If $p\ge1$, then
\[
 \delta(Q)\ge1-10p,
\]
with equality if and only if
\[
 Q\cong R^{(2)}_{5p+10}:=
 A_{5p+10}-35-(5p+6)(5p+8).
\]
\item\label{item:defect-size} If $m\ge7$, then $e(Q)\le\lfloor(5m+5p+1)/3\rfloor$. A non-$B$, pentagon-free component satisfies $m\le10\delta(Q)$. A component with $p\ge1$, on writing $\varepsilon=\delta(Q)-(1-10p)$, satisfies
\[
 m\le5p+10+3\varepsilon.
\]
\end{enumerate}
\end{theorem}
\begin{proof}
For $m\ge7$, use Theorem~\ref{thm:cover}. If $Q\subseteq A_m$ and $h=e(A_m)-e(Q)$, Lemma~\ref{lem:windows} gives
\begin{equation}\label{eq:basic}
 3h\ge m-4-5p,\qquad
 \delta(Q)=7-2m+7h\ge-1-10p+h.
\end{equation}
If $Q\subseteq W_m$, the analogous estimate is
\begin{equation}\label{eq:wrapped}
 \delta(Q)\ge6-10p+h.
\end{equation}
Both imply the edge bound in part~\textup{(\ref{item:defect-size})}.

For $p=0$, the $A_m$ estimate is
\[
 e(Q)\le2m-1-\left\lceil\frac{m-4}{3}\right\rceil
 =\left\lfloor\frac{5m+1}{3}\right\rfloor\le\frac{12m}{7}.
\]
The last inequality is strict for $m>7$. At $m=7$, equality requires deletion of exactly one edge from $A_7$. The three canonical windows have just the edge $35$ in common. Deleting it gives $B$. The wrapped estimate is strictly smaller. For $m\le6$ the ordinary pentagon-free edge maxima are at most $0,1,3,6,7,9$, respectively, all strictly below $12m/7$.

Here is a short justification of the only small bounds needed. On five vertices, any graph with at least eight edges is $K_5$ minus at most two edges and contains a Hamilton cycle. On six vertices, if a universal vertex exists its deletion is $P_4$-free, with at most four edges. If maximum degree is at most three, there are at most nine edges. Otherwise take a degree-four vertex $v$, neighbours $A$ and non-neighbour $x$. If there were at least ten edges, $e(A)+d_A(x)\ge6$. When $d_A(x)\le2$, $A$ has a $P_4$. When $d_A(x)=3$, a $P_4$-free $A$ with three edges must be a star or a triangle plus an isolated vertex, and adjoining $x$ produces a four-vertex path with both ends in $A$. When $d_A(x)=4$, any edge in $A$ produces such a path. Each case gives a pentagon through $v$.

For $p\ge1$ and $m\ge10$, $A_m$ contains two edge-disjoint three-page books. The left book has central edge $34$, with pages $1,2,5$; the right book is its reflection. Each must lose an edge, so $h\ge2$. Equation~\eqref{eq:basic} proves $\delta(Q)\ge1-10p$. For $7\le m\le9$, at least one edge must be deleted to avoid the left book, giving $\delta(Q)\ge14-2m\ge-4> -9$. For $m\le6$, necessarily $p=1$, and the elementary join edge bound $e(Q)\le2m-1$ gives $\delta(Q)\ge7-2m>-9$. The wrapped family cannot give equality, by \eqref{eq:wrapped}.

Equality in the $A_m$ case forces $h=2$ and $m=5p+10\ge15$. Exactly six canonical windows must be destroyed. Thus each deleted edge destroys three, with no overlap, and one deleted edge must belong to each end book. The only edge of the left book belonging to three canonical windows is $35$; the right one is its reflection. These deletions give $R^{(2)}_{5p+10}$. This graph consists of two $K_4$ caps and a central path square on $5p+4$ vertices. Its packing number is $p$ and its join is 1-planar, as constructed below.

Finally, for $p=0$ and $m\ge7$ the edge bound implies $\delta(Q)\ge(m-7)/3$. If $Q\not\cong B$, then $\delta(Q)\ge1$, so $m\le3\delta(Q)+7\le10\delta(Q)$. The small cases also satisfy this inequality. Substituting $\delta(Q)=1-10p+\varepsilon$ in the edge bound of part~\textup{(\ref{item:defect-size})} yields $m\le5p+10+3\varepsilon$; the cases $m\le6$ satisfy it as well.
\end{proof}

\section{Extremal constructions and edge maxima}\label{sec:constructions}
Define, for $m\ge8$,
\[
 R_m^{(1)}=P_m^2+14-35,\qquad
 R_m^{(2)}=A_m-35-(m-4)(m-2).
\]
A central path-square strip has a 1-drawing of its $K_2$-join with $zw$ uncrossed and facial triangles at both ends. To see this, draw the planar double fan on $z,w,v_1,\ldots,v_s$. Draw $v_iv_{i+2}$ across the spoke $zv_{i+1}$ for odd $i$, and across $wv_{i+1}$ for even $i$. Its two segments lie in the adjacent triangular faces at that spoke. Consecutive added edges use opposite sides of the path, while edges on the same side have disjoint face interiors. Thus no added edges cross one another, each crossed spoke is used once, and $zwv_1,zwv_s$ remain facial triangles. Each end $K_4$ cap is attached using a $K_6$ drawing with one prescribed facial triangle. The clique gluing lemma \cite[Lemma 3.1]{GHHMTZ} applies. A single $K_6$ is never required to supply two facial triangles sharing $zw$. Separate remainder components can be glued along their uncrossed copies of $zw$.

For $p\ge1$, put $M=5p+10$ and define the following admissible cores:
\begin{center}\begin{tabular}{clcc}\toprule
$E$&$C_{p,E}$&order&defect\\
\midrule
0&$R_M^{(2)}$&$M$&$1-10p$\\
1&$R_{M+3}^{(2)}-68$&$M+3$&$2-10p$\\
2&$R_{M-1}^{(2)}$&$M-1$&$3-10p$\\
3&$R_{M+2}^{(2)}-57$&$M+2$&$4-10p$\\
4&$R_{M-2}^{(2)}$&$M-2$&$5-10p$\\
5&$R_{M+1}^{(2)}-57$&$M+1$&$6-10p$\\
6&$R_{M-3}^{(1)}$&$M-3$&$7-10p$\\
\bottomrule
\end{tabular}\end{center}
Each has packing number $p$. For the three edge-deleted cases, the deleted distance-two edge makes its intermediate vertex a cut vertex; the resulting small block has order at most four, and the other central path-square block has order $5p+3$ or $5p+4$. All cycles lie within blocks. The other cases are immediate from the cap decomposition. Consequently, for every integer $q\ge0$, the graph $K_2\vee(qB\cup C_{p,E})$ is 1-planar and has at most $p+2$ disjoint pentagons.

For $p=0$, set $\mathsf D=K_4-e$ and use the block chains
\begin{align*}
 Q_8&=K_4-K_2-K_4,& Q_9&=K_4-K_3-K_4,\\
 Q_{10}&=K_4-\mathsf D-K_4,&Q_{11}&=K_4-K_3-K_3-K_4,\\
 Q_{12}&=K_4-K_3-\mathsf D-K_4,&Q_{13}&=K_4-\mathsf D-\mathsf D-K_4.
\end{align*}
Adjacent blocks share just one cut vertex; internal diamonds use their two nonadjacent degree-two vertices as ports. These are pentagon-free and have 1-planar joins with uncrossed $zw$. Their defects, in order, are $5,3,1,6,4,2$.

\begin{proof}[Proof of Corollary~\ref{cor:edge}]
If $H$ has positive-packing components with packing numbers $p_1,\ldots,p_c$, then Theorem~\ref{thm:defect} gives
\[
 D(H):=12N-7e(H)\ge c-10\sum_i p_i\ge1-10k.
\]
If there are none, $D(H)\ge0$. Also $D(H)\equiv5N\pmod7$. Choose the core $C_{k,E}$ whose defect is the least integer at least $1-10k$ in this congruence class, and pad it with copies of $B$. For $k=0$ use $B,Q_8,\ldots,Q_{13}$. These constructions attain the bounds.
\end{proof}

\section{Normalized resolvents and finite cores}\label{sec:resolvent}
For $x>6$ define
\[
 f_Q(x)=\one^\top(xI-A_Q)^{-1}\one,
 \quad W_j(Q)=\one^\top A_Q^j\one,
 \quad\Psi_j(Q)=7W_j(Q)-v(Q)W_j(B).
\]
Here $W_0(B),W_1(B),W_2(B),W_3(B)$ are $7,24,90,324$ and
\[
 f_B(x)=\frac{7x+10}{x^2-2x-6}.
\]
For $H$ of order $N$ put $\Gamma_H=7f_H-Nf_B$. The exact identity
\begin{equation}\label{eq:resolvent}
 \rho(K_2\vee H)-1=2f_H(\rho(K_2\vee H))
\end{equation}
is \cite[Lemma 2.5]{LWZ}; it also follows directly from the eigenvector equation at the two poles. Increasing $f_H$ at the old spectral root increases that root.

Let $Q_1,\ldots,Q_c$ be the components of $H$ with positive packing number. Put $p_i=\nf(Q_i)$, $P=\sum_{i=1}^c p_i=\nf(H)$, and $D=D(H)=12N-7e(H)$. Let $d_0$ be the sum of the defects of the pentagon-free components, and put $\varepsilon_i=\delta(Q_i)-(1-10p_i)$ for $1\le i\le c$. Theorem~\ref{thm:defect} implies
\[
 D=c-10P+\sum_i\varepsilon_i+d_0,
 \qquad \eta:=D+11P=c+P+\sum_i\varepsilon_i+d_0\ge0.
\]
The total number $S$ of vertices in non-$B$ components satisfies
\begin{equation}\label{eq:mass}
 S\le5P+10c+3\sum_i\varepsilon_i+10d_0
 \le10D+105P\le10\eta.
\end{equation}
For $P=0$ the same assertion is $S\le10D$.
Since $\Delta(H)\le6$, $W_j(Q)\le v(Q)6^j$. All $B$ components cancel identically in $\Gamma_H$. Consequently
\begin{equation}\label{eq:tail}
 \left|\Gamma_H(x)+\frac{2D}{x^2}\right|
 \le\frac{5040\eta}{x^3(1-6/x)}.
\end{equation}
Thus the bound depends only on the defect and the packing number.

\begin{theorem}[Minimum defect and finite core]\label{thm:kernel}
Let $t\ge4$, $k=t-3$, and
\[
 n\ge N_{\mathrm{red}}(t):=\max\{10^{19},250t,(100000+3000k)^2\}.
\]
Write $N=n-2$ and let $E\in\{0,\ldots,6\}$ be determined by
\[
 D_*:=1-10k+E\equiv5N\pmod7.
\]
Every graph in $\SPEX(n,tC_5)$ has the form $K_2\vee H$, where $H=qB\cup F$ for some integer $q\ge0$, $F$ has no component isomorphic to $B$, and
\begin{gather*}
 D(H)=D_*,\qquad \nf(F)=k,\\
 v(F)\le5k+10(E+1)\le5k+70=5t+55.
\end{gather*}
If $c$ is the number of positive-packing components of $F$, their excesses and the pentagon-free defect satisfy
\begin{equation}\label{eq:budget}
 E=(c-1)+\sum_i\varepsilon_i+d_0.
\end{equation}
\end{theorem}
\begin{proof}
Use Proposition~\ref{prop:reduction}. Corollary~\ref{cor:edge} gives $D\ge D_*$. Put $q_*=(N-v(C_{k,E}))/7$. Since $\delta(C_{k,E})=D_*\equiv5N\pmod7$ and $\delta(C_{k,E})\equiv5v(C_{k,E})\pmod7$, $q_*$ is an integer. The bound on $n$ gives $N\ge5k+13\ge v(C_{k,E})$, so $q_*\ge0$. The admissible comparison graph $H_*=q_*B\cup C_{k,E}$ has defect $D_*$. If $D>D_*$, put $\Delta_D:=D-D_*\ge7$. Since $P\le k$, \eqref{eq:tail} gives
\[
 \Gamma_{H_*}(x)-\Gamma_H(x)
 \ge\frac{2\Delta_D}{x^2}
 -\frac{5040[\Delta_D+2(D_*+11k)]}{x^3(1-6/x)}.
\]
As $D_*+11k\le k+7$, this is positive when $x>7566+720k$, in particular at $x=\rho(G)$ under the stated bound. Equation~\eqref{eq:resolvent} contradicts extremality. Hence $D=D_*$.

If $P=0$, then $D\ge0$, whereas $D_*\le7-10k<0$, a contradiction. Thus $c\ge1$. Rearranging the defect identity gives
\[
 E=(c-1)+10(k-P)+\sum_i\varepsilon_i+d_0.
\]
Every term is nonnegative, while $E\le6$, so $P=k$. This proves \eqref{eq:budget}. Now
\[
 v(F)\le5k+10c+3\sum_i\varepsilon_i+10d_0
 \le5k+10(E+1).
\]
\end{proof}

\begin{corollary}[The zero-excess residue]\label{cor:zero}
Under the hypotheses of Theorem~\ref{thm:kernel}, if $N\equiv5k+10\pmod7$, then the unique extremal graph is
\[
 K_2\vee\left(\frac{N-5k-10}{7}B\ \cup R_{5k+10}^{(2)}\right).
\]
\end{corollary}
\begin{proof}
Here $E=0$. Equation~\eqref{eq:budget} forces one positive-packing component, zero excess, and no positive defect in the pentagon-free part. The equality cases of Theorem~\ref{thm:defect} determine both parts uniquely.
\end{proof}

\section{Pentagon-free components and the boundary case}\label{sec:t3}
For $t=3$, $H$ is pentagon-free. Let $r\in\{0,\ldots,6\}$ satisfy $N\equiv r\pmod7$. The same resolvent comparison, now with $\eta=D$, gives
\[
 D(H)=\tau_r,\qquad (\tau_0,\tau_1,\ldots,\tau_6)=(0,5,3,1,6,4,2).
\]
Indeed a larger defect differs by at least seven, and \eqref{eq:tail} makes the admissible minimum-defect competitor strictly better for $\rho>6846$. Thus all non-$B$ components together have at most 60 vertices.

\begin{lemma}[Finite component bound]\label{lem:finitecomponents}
Let $Q$ be a connected pentagon-free graph, not isomorphic to $B$, for which $K_2\vee Q$ is 1-planar, and suppose $d=\delta(Q)\in\{1,\ldots,6\}$. Then
\[
 \Psi_2(Q)\le-b_d,\qquad
 (b_1,\ldots,b_6)=(46,78,54,86,118,94).
\]
The pair $(\Psi_2(Q),\Psi_3(Q))$ is lexicographically maximized, uniquely up to isomorphism, by the respective graphs
\[
 Q_{10},\quad Q_{13},\quad Q_9,\quad Q_{12},\quad Q_8,\quad Q_{11}.
\] Moreover, if $d=d_1+\cdots+d_s\le6$, with $s\ge2$ and all $d_i\ge1$, then $b_d<\sum_i b_{d_i}$.
\end{lemma}
\begin{proof}[Proof by exact finite enumeration]
For $m\ge7$, Theorem~\ref{thm:defect} gives $m\le3d+7\le25$. The zero-defect graph $B$ is included in the table below as a check on the boundary of the enumeration. A connected non-$B$ component under consideration has defect at most six. If its order is at most six, it must be $K_4$, of defect six, by the small bounds above. For order at least seven the wrapped family cannot occur: \eqref{eq:wrapped} would force $h=0$, but $W_m$ has a pentagon. In $A_m$ one has
\[
 \delta=7-2m+7h,\quad
 h\ge\lceil(m-4)/3\rceil,\quad 0\le\delta\le6.
\]
The only possible orders are
\[
 7,8,9,10,11,12,13,15,16,18,19,22,25.
\]
The following table records all fixed-budget pentagon-hitting deletion sets in these upper envelopes; the entries count labelled deletion sets. Four contain a forbidden three-page book and are excluded by \cite[Lemma 3.5]{CHM}.

\begin{center}\small\begin{tabular}{rrrrrr}\toprule
$m$&$\delta$&all sets&book-free sets&best $\Psi_2$&best $\Psi_3$ at $\Psi_2$ maximum\\
\midrule
7&0&1&1&0&0\\
8&5&5&3&$-118$&$-604$\\
9&3&1&1&$-54$&$-270$\\
10&1&1&1&$-46$&$-258$\\
11&6&9&7&$-94$&$-484$\\
12&4&2&2&$-86$&$-458$\\
13&2&1&1&$-78$&$-432$\\
15&5&3&3&$-118$&$-632$\\
16&3&1&1&$-110$&$-606$\\
18&6&4&4&$-150$&$-806$\\
19&4&1&1&$-142$&$-780$\\
22&5&1&1&$-174$&$-954$\\
25&6&1&1&$-206$&$-1128$\\
\bottomrule
\end{tabular}\end{center}

\smallskip
\noindent\emph{Verification of the finite assertion.}
List the edges of $A_m$ and encode each pentagon by the bit mask of its five edges. A deletion set is valid exactly when its mask intersects every cycle mask. Its size $h$ is fixed by the defect.

The first exhaustive algorithm selects an unhit cycle and branches on each of its five edges. Every hitting set extending the current choice must contain at least one of these edges. Branches with too many edges are discarded; a cardinality bound using the maximum number of newly hit cycles prunes only impossible branches. The second algorithm processes cycle masks in reverse order, adjoining an edge whenever the current set misses the next cycle, and retains all masks of size at most $h$. In either algorithm, a completed hitting set of size less than $h$ is extended by every choice of unused edges needed to reach size $h$. Both algorithms return the same 31 sets. A third implementation branches directly on keeping or deleting each edge and returns the same family. The accompanying JSON certificate records every set and its exact integer moments, and explicit vertex bijections resolve all ties at the optimal moment pair.

The two order-eight deletion sets $\{12,46\}$ and $\{35,78\}$ have $\Psi_2=-76$, larger than $Q_8$'s $-118$, but both contain a three-page book and are therefore inadmissible.

For each positive defect $d\le6$, the unique best component up to isomorphism is
\begin{center}\begin{tabular}{crrr}\toprule
$d$&component&$\Psi_2$&$\Psi_3$\\
\midrule
1&$Q_{10}$&$-46$&$-258$\\
2&$Q_{13}$&$-78$&$-432$\\
3&$Q_9$&$-54$&$-270$\\
4&$Q_{12}$&$-86$&$-458$\\
5&$Q_8$&$-118$&$-604$\\
6&$Q_{11}$&$-94$&$-484$\\
\bottomrule
\end{tabular}\end{center}
For every partition of $d\le6$ into two or more positive parts, the sum of the corresponding best $\Psi_2$ values is strictly less than the single-core value. This is a finite six-budget check, included in the exact verifier. The only remaining same-$\Psi_2$ competition between single components is defect five at orders eight and fifteen; the $\Psi_3$ gap is 28 in favour of $Q_8$.
Explicit graph-isomorphism checks identify every candidate tying the winning pair $(\Psi_2,\Psi_3)$ with the stated core. Together with the constructions in Section~\ref{sec:constructions}, this proves the finite assertions.
\end{proof}

\begin{proof}[Proof of Theorem~\ref{thm:t3main}]
The covering theorem, defect theorem, and minimum-defect argument reduce every extremizer to the finite situation just classified. By the last assertion of Lemma~\ref{lem:finitecomponents}, distributing a positive defect among two or more non-$B$ components is strictly inferior already at $\Psi_2$. Hence the exceptional part is a single component of defect $\tau_r$ (unless $\tau_r=0$, when it is empty). Since the $B$ parts cancel in $\Gamma$, the remaining moment comparison depends only on that exceptional component.

Strict moment order implies strict resolvent order under the following explicit bound. Each exceptional part has at most 60 vertices. Thus for two candidates, the absolute difference of the normalized $j$th moments is at most $1680\,6^j$. The first strict gap occurs at $j=2$, with gap at least 14, or at $j=3$, with gap 28. The geometric tail is dominated for $x>77766$ in either case. Since $\rho>\sqrt{2n}$, this holds. Equation~\eqref{eq:resolvent} gives a strictly larger spectral radius unless the exceptional component is the indicated one. Isomorphism checks in the certificate establish the equality case; the bulk has zero defect only when every component is $B$. The constructions above establish admissibility.
\end{proof}

\section{A neighbourhood obstruction}\label{sec:obstruction}
We also use a necessary condition stronger than three-page-book exclusion.
\begin{lemma}[Neighbourhood obstruction]\label{lem:fan}
Suppose $K_2\vee Q$ is 1-planar. If $v\in Q$ has at least five neighbours, then $Q[N_Q(v)]$ contains no $P_4$ as a subgraph. In particular, $Q$ contains no $K_1\vee P_5$.
\end{lemma}
\begin{proof}
Czap--Hud\'ak--Madaras \cite[Lemma 2.8]{CHM} prove that $K_3\vee(P_4\cup K_1)$ is not 1-planar. A $P_4$ in $N_Q(v)$ and a fifth neighbour, together with $v,z,w$, would give this forbidden subgraph; unwanted edges may be discarded. The final assertion follows since $P_5$ contains $P_4\cup K_1$ as a non-induced subgraph.
\end{proof}
The obstruction applies uniformly to the following two endpoint families. For $m=5p+7$, $p\ge1$, let
\[
 T_m^{(0)}=A_m-35-(m-2)m,\qquad
 T_m^{(1)}=A_m-35-(m-2)(m-1).
\]
On the last six vertices, relabelled $1,\ldots,6$, vertex 3 is adjacent to all other five. The other five form respectively the paths
\[
 1,2,4,5,6\qquad\hbox{and}\qquad1,2,4,6,5.
\]
Thus both graphs contain $K_1\vee P_5$, independently of $p$. The six-vertex witnesses therefore exclude these families for every $p\ge1$.

\section{A second-moment deletion inequality}\label{sec:deletion}
For $m\ge12$ the degree sequence of $A_m$ is
\[
 (3,3,4,5,\underbrace{4,\ldots,4}_{m-8},5,4,3,3),
 \qquad W_2(A_m)=16m-10.
\]
For an edge $e$ let $c(e)$ be its number of canonical windows, and put
$b(e)=2\sum_{v\in e}d_{A_m}(v)-2$.
\begin{lemma}[Window--moment inequality]\label{lem:windowmoment}
Delete a set $F$ of $h$ edges from $A_m$, and let $u$ be the number of canonical windows destroyed. Then
\[
 L:=W_2(A_m)-W_2(A_m-F)\ge4u+2h.
\]
\end{lemma}
\begin{proof}
Delete edges in any order. Suppose the current edge $ab$ lies in $q$ still-present canonical pentagons. In those pentagons, the other neighbours of $a$ are pairwise distinct, and the same is true for $b$. For example, if $ab=j(j+2)$, the possible window starts are $j-2,j-1,j$; the other neighbours of $j$ are $j-2,j-1,j+1$, and those of $j+2$ are $j+1,j+3,j+4$. For $ab=j(j+1)$ the two possible starts are $j-3,j$, again giving distinct other neighbours. An extra end chord lies in no canonical window. Consequently the current degrees at $a,b$ are each at least $q+1$. Removing $ab$ decreases the sum of squared degrees by
\[
 2d(a)+2d(b)-2\ge4q+2.
\]
Each destroyed window is counted exactly at its first removed edge. Summing proves the claim.
\end{proof}
For later exact checks, if $r_v$ counts deleted edges incident with $v$, then
\begin{equation}\label{eq:loss}
 L=\sum_v(2d_{A_m}(v)r_v-r_v^2)
   =\sum_{e\in F}b(e)-2\sum_v\binom{r_v}{2}.
\end{equation}

\section{Low-excess classification}\label{sec:local}
Write
\[
 (c_0,c_1,\ldots,c_6)=(46,78,68,86,90,108,112).
\]
\begin{theorem}[Local lexicographic ordering]\label{thm:local}
Let $Q$ be connected with a 1-planar $K_2$-join, let $p=\nf(Q)\ge1$, and let
\[
 \varepsilon=\delta(Q)-(1-10p)\in\{0,\ldots,6\}.
\]
Then
\[
 \Psi_2(Q)\le110p-c_\varepsilon,
 \qquad
 \Psi_2(C_{p,\varepsilon})=110p-c_\varepsilon.
\]
Moreover the normalized moment sequence $(\Psi_2,\ldots,\Psi_5)$ is lexicographically maximized uniquely, up to isomorphism, by $C_{p,\varepsilon}$ of Section~\ref{sec:constructions}.
\end{theorem}
The proof uses the covering theorem, the two necessary obstructions, and bounded local case analysis within the book-free relaxation; admissibility of the maximizing cores is supplied by Section~\ref{sec:constructions}.

\subsection{Small orders and the deletion budget}
For order at most five, a positive-packing graph has order five; nine edges already force a three-page book. On six vertices a graph with a 1-planar $K_2$-join has at most ten edges. Indeed, if it has a universal vertex and at least eleven edges, the five-vertex neighbourhood has at least six edges and contains $P_4$, contrary to Lemma~\ref{lem:fan}. Otherwise its complement has minimum degree at least one and at most four edges. An isolated $K_2$ component in the complement gives $K_{4,2}$ in the graph, forbidden by \cite[Lemma 3.5]{CHM}. With no such component the complement must be $2P_3$; the two leaves of one $P_3$ give an edge with three common neighbours in the graph. This is also forbidden. Thus these orders cannot have $\varepsilon\le6$.

For the augmented-envelope case with $m\ge7$, set $F:=E(A_m)\setminus E(Q)$ and $h:=|F|=e(A_m)-e(Q)$. Let $u$ be the number of canonical windows destroyed by $F$, and put $L:=W_2(A_m)-W_2(Q)$. For orders seven and eight, at least one edge must be deleted from $A_m$ to avoid a book, which again excludes low positive-packing excess. For order nine the two end books are edge-disjoint, so at least two deletions are necessary. The only possible low-excess order-nine case is the wrapped graph treated below. For $m=10,11$, the two end books are edge-disjoint, so $h\ge2$ and
\[
 \varepsilon=6+10p-2m+7h\ge30-2m\in\{10,8\},
\]
contrary to $\varepsilon\le6$. In the augmented case, therefore, $m\ge12$ and there must be at least two deletions.

The defect equation and window bound give
\begin{equation}\label{eq:heps}
 m=5p+3+\frac{7h-\varepsilon}{2},\qquad
 2\le h\le\varepsilon+2,\qquad h\equiv\varepsilon\pmod2,
 \qquad u\ge\frac{7h-\varepsilon-2}{2}.
\end{equation}
By Lemma~\ref{lem:windowmoment},
\begin{equation}\label{eq:roughpsi}
 \Psi_2(Q)=110p-4+77h-11\varepsilon-7L
 \le110p+24-35h+3\varepsilon.
\end{equation}
Every permitted $h\ge4$ is already strictly inferior to the claimed row, except $(\varepsilon,h)=(6,4)$. Thus only two deletions, three deletions, and this one four-deletion case require further attention.

\subsection{Two deletions}
The left end book has the following seven edges. Here $I(e)$ is the set of starting positions of the canonical windows containing $e$, so $c(e)=|I(e)|$. The right table is its reflection; bars below denote $i\mapsto m+1-i$.
\begin{center}\begin{tabular}{c|rrrrrrr}\toprule
$e$&13&14&23&24&34&35&45\\
\midrule
$I(e)$&$\{1\}$&$\varnothing$&$\{2\}$&$\{1,2\}$&$\{3\}$&$\{1,2,3\}$&$\{1,4\}$ \\
$c(e)$&1&0&1&2&1&3&2\\
$b(e)$&12&14&12&14&16&14&16\\
\bottomrule
\end{tabular}\end{center}
With just two deletions, one must occur in each book. The two ends have disjoint incident vertex sets. Deleting only 13 or only 23 at the left end leaves $K_1\vee P_5$ on positions $1,\ldots,6$, and the distant deletion cannot remove it. These choices are forbidden. Each allowed end deletion therefore costs at least 14, and equality requires one of $14,24,35$. Deleting 24 is isomorphic to deleting 14 by exchanging vertices 1 and 2; neither has other neighbours outside the first four vertices. Thus every second-moment maximizer is isomorphic to one of
\[
 P_m^2,\qquad R_m^{(1)},\qquad R_m^{(2)}.
\]
Their pentagon packing numbers are respectively
\[
 \lfloor m/5\rfloor,\qquad\lfloor(m-3)/5\rfloor,
 \qquad\lfloor(m-6)/5\rfloor.
\]
Since $m=5p+10-\varepsilon/2$, for $\varepsilon=0,2,4$ only the last is admissible. For $\varepsilon=6$, $R_m^{(1)}$ and $R_m^{(2)}$ are admissible. Directly from $W_3=2\sum_{ab\in E}d(a)d(b)$, for $m\ge12$,
\[
 W_3(R_m^{(1)})=64m-206,\qquad
 W_3(R_m^{(2)})=64m-212.
\]
Thus the one-cap graph wins strictly in the last row. In all four rows, $L=28$ gives $\Psi_2=110p-46-11\varepsilon$, exactly as asserted.

\subsection{Three deletions: a bounded local table}
Here $m\ge16$ and $\varepsilon=1,3,5$. The required values of $u$ are 9, 8, 7. Using the preceding seven-edge table and \eqref{eq:loss}, with the book and neighbourhood obstructions, gives the following complete equality table. Reflections are understood.
\begin{center}\small\begin{tabular}{ccl}\toprule
required $u$&least $L$&all equality patterns\\
\midrule
$u\ge9$&42&$\{35,\overline{35},j(j+2)\}$, $6\le j\le m-7$\\
$u\ge8$&40&$\{35,57,\overline{35}\}$\\
$u\ge7$&40&$\{35,57,\overline{35}\}$; $\{35,56,\overline{35}\}$;\\
&&$\{35,57,\overline{24}\}$\\
\bottomrule
\end{tabular}\end{center}
Here is a complete reduction to the displayed cases. The two end books have disjoint vertex sets. A third edge outside both books cannot meet a deleted edge at both ends, because each edge of $A_m$ has length at most two except the two end chords, which belong to their books. The only edges outside the left book meeting one of its vertices are
\[
 12,\quad46,\quad56,\quad57.
\]
The reflected statement holds at the right. All other third edges have no endpoint interaction with either end deletion and have cost at least 14.

For $u\ge9$, each of the three deletions must destroy three windows, and the three sets of windows must be disjoint. Exactly one deletion lies in each book, since the union of two book edges contains at most four windows. The only book edge of capacity three is $35$. The third deletion is therefore $j(j+2)$ with $6\le j\le m-7$. All endpoint sets are disjoint and each cost is 14, giving $L=42$.

For $u\ge7$, first consider two deletions in the same book. The only pair hitting four different windows is $35,45$, with window set $\{1,2,3,4\}$ and joint loss $14+16-2=28$. Every other such pair hits at most three windows. The lone deletion at the opposite end cannot be $13$ or $23$ (reflected), since its neighbourhood obstruction would be untouched. It has cost at least 14, so this case has $L\ge42$ and cannot attain 40.

It remains to take one deletion in each book and a third edge outside them. Deleting $13$ alone leaves, in the neighbourhood of vertex 4, the path $1,2,3,5$ and the fifth neighbour 6. Deleting $23$ alone leaves the path $1,3,5,6$ and the fifth neighbour 2. A third edge away from this end cannot remove these witnesses. Among the four nearby choices, $12$ and $56$ also leave a four-vertex path in the same neighbourhood, and $57$ leaves the displayed witnesses unchanged. Only $46$ can repair the obstruction. The pair $13,46$ has loss 28, and $23,46$ destroys only three windows. Thus a bad lone end never gives $u\ge7$ and $L\le40$.

Table~\ref{tab:localpairs} lists every nearby pair $(e,g)$, with $e$ in the left book and $g\in\{12,46,56,57\}$, whose union destroys at least four windows. Pairs hitting at most three windows cannot reach seven after the opposite deletion. The final column records the neighbourhood obstruction used in the reduction.

\begin{table}[H]
\centering\small
\begin{tabular}{cccll}
\toprule
$e$ & $g$ & $I(e)\cup I(g)$ & pair loss & obstruction \\
\midrule
13&46&$\{1,2,3,4\}$&28&--\\
13&57&$\{1,3,4,5\}$&26&neighbourhood $P_4$\\
23&57&$\{2,3,4,5\}$&26&neighbourhood $P_4$\\
24&46&$\{1,2,3,4\}$&28&--\\
24&57&$\{1,2,3,4,5\}$&28&--\\
35&46&$\{1,2,3,4\}$&30&--\\
35&56&$\{1,2,3,5\}$&26&--\\
35&57&$\{1,2,3,4,5\}$&26&--\\
45&46&$\{1,2,3,4\}$&30&--\\
45&56&$\{1,2,4,5\}$&28&--\\
45&57&$\{1,3,4,5\}$&28&--\\
\bottomrule
\end{tabular}
\caption{Exact local exclusions for three deletions. Reflections give the right-end cases. The pair loss is calculated by~\eqref{eq:loss}; the surviving opposite-end deletion costs at least 14.}\label{tab:localpairs}
\end{table}

After the bad-end cases just excluded, if the third deletion has no endpoint interaction, each permissible end deletion costs at least 14 and $L\ge42$. If there is an interaction, the table shows that $L\le40$ forces $\{e,g\}=\{35,56\}$ or $\{35,57\}$, of loss 26, with a cost-14 opposite deletion. The latter is one of $\overline{14},\overline{24},\overline{35}$; it hits respectively zero, two, or three windows. Therefore $35,56$ can reach seven only with $\overline{35}$, while $35,57$ reaches eight with $\overline{35}$ and seven with $\overline{24}$. These are exactly the three listed equality patterns. This proves both the lower bounds and the completeness of their equality cases.
For $\varepsilon=3$ the unique pattern is the claimed $C_{p,3}$. For $\varepsilon=5$, the pattern with 56 leaves the disjoint canonical windows starting at $4,9,\ldots,5p+4$, so its packing exceeds $p$. The pattern with $\overline{24}$ is isomorphic, at that end, to a plain path-square tail. After the cut created by 57, that tail has $m-5=5p+6$ vertices, also giving at least $p+1$ pentagons. Only $C_{p,5}$ and its reflection remain. The resulting costs are $c_3=86$ and $c_5=108$.

For $\varepsilon=1$ the cut created by $j(j+2)$ divides the central strip into path squares of orders $a=j-2$ and $b=m-j-3$, sharing their endpoint. Here $a+b=5p+8$. Their packing is $p$ exactly when $a\equiv b\equiv4\pmod5$; otherwise it is $p+1$. To justify attainability of the latter floor sum, at least one of $a,b$ is not divisible by five, so a maximum packing on that side can avoid the shared endpoint. Hence the allowed positions are
\[
 j=6+5s,\qquad 0\le s\le p.
\]
All have the same second and third moments, but the fourth moment singles out the two reflected end positions.

\begin{lemma}[Fourth-moment tie]\label{lem:fourth}
In $R_m^{(2)}$, $m=5p+13$, deleting $(6+5s)(8+5s)$ decreases $(W_2,W_3,W_4)$ by
\[
 \begin{cases}(14,80,398),&s=0\text{ or }s=p,\\
 (14,80,408),&0<s<p.
 \end{cases}
\]
\end{lemma}
\begin{proof}
Let $A$ be the adjacency matrix of $R_m^{(2)}$. The endpoints have degree four. Put $d=A\one$, $a=A d$, and let $b$ be the sum of their unit vectors. After deletion,
\[
 d'=d-b,\qquad a'=a-c,\qquad c=A b+3b.
\]
The two endpoints have $a$-value 16, so the third-moment loss is $2(16+16)+16=80$. The seven nonzero coordinates of $c$, in their strip order, are
\[
 (1,1,4,2,4,1,1),\qquad \|c\|^2=40.
\]
At an interior allowed position, every corresponding $a$-coordinate is 16, giving $2a^Tc-\|c\|^2=408$. At the left end, $a_5=13$ and $a_7=15$ instead of 16, weighted by $c_5=1,c_7=2$; all other relevant values are 16. The loss decreases by $2(3+2)=10$. The right end is its reflection. Allowed interior positions are separated from both exceptional end neighbourhoods, also when $p=2$.
\end{proof}

\subsection{The remaining four-deletion case}
For $(\varepsilon,h)=(6,4)$, $m=5p+14\ge19$ and $u\ge10$. In fact
\begin{equation}\label{eq:fourloss}
 L\ge52.
\end{equation}
To see this, write $C_F=\sum_{e\in F}c(e)$ and $\ell=C_F-u$. Then
$\sum_e(3-c(e))+\ell=12-u\le2$.
An edge of capacity three is a distance-two edge and costs at least 14. Two such edges sharing a vertex overlap in a canonical window; the overlap positions for successive distance-two edges are distinct. If all four edges have capacity three, the number of shared-endpoint pairs is at most $\ell\le2$, and \eqref{eq:loss} gives $L\ge56-4$.

If one edge has capacity two and the others three, every edge still costs at least 14. A capacity-two edge sharing a vertex with a capacity-three edge overlaps a window. There is at most one overlap, so $L\ge54$. If two edges have capacity two, there is no overlap allowance. Only those two edges, if both are consecutive distance-one edges, can share a vertex without overlap; thus $L\ge54$. The last possibility is one capacity-one edge and three capacity-three edges, with no overlaps. The capacity-one edge costs at least 10, and it cannot share a vertex with a capacity-three edge without overlap. Hence $L\ge10+3\cdot14=52$. These cases exhaust the deficit budget of two. End chords have capacity zero and cannot occur. Equation~\eqref{eq:fourloss} gives
\[
 \Psi_2(Q)\le110p+238-7\cdot52=110p-126,
\]
strictly below the excess-six candidate. This completes the augmented-envelope cases.

\section{Wrapped competitors and the fifth moment}\label{sec:wrapped}
For a wrapped-envelope subgraph, the global window bound gives $h\le\varepsilon-5$. Parity and the defect equation leave only
\[
 (\varepsilon,h,m)=(5,0,5p+4),\qquad(6,1,5p+7).
\]
In the second case $W_2(W_m)=16m-28$. The four degree-three vertices are $1,2,m-1,m$; all others have degree four. To tie or exceed $110p-112$, the deleted edge must have both ends of degree three. It is one of $12,(m-1)m,1m$. The last deletion leaves $P_m^2$, with $p+1$ disjoint pentagons. Deleting either of the other two still leaves $p+1$ disjoint consecutive five-vertex windows. Thus every admissible excess-six wrapped remainder loses already at the second moment.

For excess five, put $U_p=W_{5p+4}$ and $V_p=C_{p,5}$. They have different orders, but $U_p\cup B$ and $V_p$ have the same order.
\begin{lemma}[Fifth-moment identity]\label{lem:fifth}
For every $p\ge1$, writing $\iota=\mathbf1_{\{p=1\}}$, the moments of $V_p$ and $U_p\cup B$ are
\begin{center}\begin{tabular}{c|ll}\toprule
$j$&$W_j(V_p)$&$W_j(U_p\cup B)$\\
\midrule
0&$5p+11$&$5p+11$\\
1&$20p+36$&$20p+36$\\
2&$80p+126$&$80p+126$\\
3&$320p+426$&$320p+426$\\
4&$1280p+1448$&$1280p+1448$\\
5&$5120p+4836+4\iota$&$5120p+4820+2\iota$\\
\bottomrule
\end{tabular}\end{center}
In particular the first strict difference is $18$ for $p=1$, and $16$ for every $p\ge2$.
\end{lemma}
\begin{proof}
For a graph with adjacency matrix $A$, put $d=A\one$ and $s=A d$. Then
\[
 W_2=d^Td,\quad W_3=2\sum_{ab\in E}d_a d_b,\quad
 W_4=s^Ts,\quad W_5=2\sum_{ab\in E}s_a s_b.
\]
For $U_p$, of order $l=5p+4$, the $s$-values at positions $1,2,3,4$ are $10,11,14,15$, with the reflected values at the right end and 16 elsewhere. Write $s=16\one+e$. The degree-weighted deviation sum is $\sum d_v e_v=-90$. The products $e_a e_b$ on edges sum to 59 at each end, plus 36 on the wrap edge, and an additional 1 when $l=9$. The additional edge is the distance-two edge joining the two $-1$ boundary deviations. Therefore
\[
 W_5(U_p)=2[256(2l-2)-1440+154+\iota]
         =1024l-3596+2\iota.
\]
Also $W_5(B)=4320$.

For $V_p$, of order $m=5p+11$, the left $s$-profile at positions 1 through 9 is
\[
 (11,11,11,15,9,14,12,15,15),
\]
and the right profile at positions $m-6$ through $m$ is
\[
 (15,16,13,16,11,11,11).
\]
All other coordinates are 16. These profiles are disjoint even for $p=1$. Here $\sum d_v e_v=-150$. The edge products of deviations contribute 132 on the left and 78 on the right. For $m=16$ two additional edges join a left $-1$ deviation to the right $-1$ deviation, adding 2. For $m\ge21$ there is no such cross-term. Since $e(V_p)=2m-4$,
\[
 W_5(V_p)=2[256(2m-4)-2400+210+2\iota]
         =1024m-6428+4\iota.
\]
Substitution yields the last row, and the preceding rows follow from the same degree and $s$ profiles.
\end{proof}
This finishes the proof of Theorem~\ref{thm:local}. The wrapped component is a comparison graph in the normalized resolvent order.

\section{Defect condensation and the final classification}\label{sec:condensation}
Let $b_d$, $1\le d\le6$, be the costs from the shared exact finite Lemma~\ref{lem:finitecomponents}:
\[
 (b_1,\ldots,b_6)=(46,78,54,86,118,94).
\]
The following elementary inequalities hold whenever the indicated indices do not exceed six:
\begin{align}
 c_a+c_b&>c_{a+b+1},&&a,b\ge0,\ a+b+1\le6,\label{eq:merge}\\
 c_a+b_d&>c_{a+d},&&a\ge0,\ d\ge1,\ a+d\le6.\label{eq:absorb}
\end{align}
They are 12 and 21 finite integer inequalities, respectively, and each gap is at least 14. Under the defect identity~\eqref{eq:budget}, splitting positive packing between components, or retaining positive pentagon-free defect outside the packing component, incurs a strictly larger normalized second-moment penalty.

\begin{proof}[Proof of Theorem~\ref{thm:structure}]
Theorem~\ref{thm:kernel} applies because $N_0(t)\ge N_{\mathrm{red}}(t)$. Thus every extremal remainder has minimum defect $D^*=1-10k+E$, uses all $k$ units of pentagon packing, and has bounded non-$B$ part. Write the positive-packing components as $Q_1,\ldots,Q_c$, with
\[
 p_i=\nu_5(Q_i),\qquad
 \varepsilon_i=\delta(Q_i)-(1-10p_i),
\]
and write the pentagon-free non-$B$ components as $R_1,\ldots,R_s$, with
\[
 d_j=\delta(R_j)\in\{1,\ldots,6\}.
\]
Then $\sum_i p_i=k$, and equation~\eqref{eq:budget} becomes
\begin{equation}\label{eq:budget-expanded}
 E=(c-1)+\sum_{i=1}^c\varepsilon_i+\sum_{j=1}^s d_j.
\end{equation}
In particular all indices appearing below are at most six.

The normalized moments are additive over disjoint unions, and every $B$-component contributes zero to every $\Psi_j$. Hence Theorem~\ref{thm:local} and Lemma~\ref{lem:finitecomponents} give
\begin{equation}\label{eq:global-penalty}
 \Psi_2(H)
 \le 110k-
 \left(\sum_{i=1}^c c_{\varepsilon_i}
       +\sum_{j=1}^s b_{d_j}\right).
\end{equation}
Now merge the positive-packing penalties successively using~\eqref{eq:merge}. After all $c$ terms have been merged, their total penalty is strictly larger than
\[
 c_A,\qquad
 A=(c-1)+\sum_{i=1}^c\varepsilon_i,
\]
unless $c=1$. Each intermediate index is at most $A\le E\le6$. Next absorb the pentagon-free defect terms one at a time using~\eqref{eq:absorb}. By~\eqref{eq:budget-expanded}, the final index is exactly $E$. Consequently
\begin{equation}\label{eq:penalty-condensation}
 \sum_{i=1}^c c_{\varepsilon_i}+\sum_{j=1}^s b_{d_j}
 \ge c_E,
\end{equation}
with strict inequality whenever $c>1$ or $s>0$.

Consider the globally admissible comparison remainder
\[
 H_*=q_*B\cup C_{k,E},
 \qquad q_*:=\frac{N-v(C_{k,E})}{7}.
\]
The congruence argument in the proof of Theorem~\ref{thm:kernel} shows that $q_*$ is a nonnegative integer, and Section~\ref{sec:constructions} proves that $K_2\vee H_*$ is 1-planar and $tC_5$-free. By Theorem~\ref{thm:local},
\[
 \Psi_2(H_*)=\Psi_2(C_{k,E})=110k-c_E.
\]
Thus~\eqref{eq:global-penalty}--\eqref{eq:penalty-condensation} imply
\[
 \Psi_2(H)<\Psi_2(H_*)
\]
whenever $c>1$ or $s>0$. Since $H_*$ is independently admissible, the comparison is between the complete remainders $H$ and $H_*$.

The strict second-moment gap also gives a strict spectral gap in the present range of $n$. Before condensation, Theorem~\ref{thm:kernel} bounds the total non-$B$ part by $5k+70$ vertices; the core of $H_*$ satisfies the same bound. Hence, after the $B$-components cancel, for the two exceptional parts $F$ and $C_{k,E}$ and every $j\ge1$,
\[
 |\Psi_j(F)-\Psi_j(C_{k,E})|\le K6^j,
 \qquad K=28(5k+70)=140(k+14).
\]
The first strict gap above occurs at $j=2$ and is at least $14$, while the remaining normalized resolvent tail is bounded geometrically. The definition of $N_0(t)$ and $\rho>\sqrt{2n}$ make that tail smaller than the leading gap. Equation~\eqref{eq:resolvent} therefore gives
\[
 \rho(K_2\vee H_*)>\rho(K_2\vee H),
\]
contradicting extremality. Hence necessarily $c=1$ and $s=0$.

Let $Q=Q_1$. Equation~\eqref{eq:budget-expanded} now gives $\varepsilon_1=E$, so
\[
 H=qB\cup Q,\qquad \nu_5(Q)=k,\qquad \delta(Q)=1-10k+E.
\]
The size estimate in Theorem~\ref{thm:defect} yields
\[
 v(Q)\le5k+10+3E\le5k+28.
\]
Thus $Q$ is the only non-$B$ component and carries all packing and all nonzero defect, as claimed.
\end{proof}

\begin{proof}[Proof of Theorem~\ref{thm:main}]
By Theorem~\ref{thm:structure}, every extremal remainder is $qB\cup Q$, where $Q$ has packing number $k$ and excess $E$. Theorem~\ref{thm:local} determines the normalized moment sequence $(\Psi_2,\ldots,\Psi_5)$ uniquely: it is lexicographically maximized by $C_{k,E}$, with the first strict difference occurring by moment five.

To transfer this finite moment comparison to the actual spectral root, let $F_1,F_2$ be two possible cores. Theorem~\ref{thm:structure} gives $v(F_i)\le5k+28$, and the coarser bound $5k+70$ is enough for the uniform estimate
\[
 |\Psi_j(F_1)-\Psi_j(F_2)|\le K6^j,
 \qquad K=140(k+14).
\]
If $\ell\le5$ is the first strict moment index, its integer gap is at least one. The rest of the normalized resolvent series has absolute value at most
\[
 \frac{K6^{\ell+1}}{x^{\ell+2}(1-6/x)},
\]
which is smaller than $x^{-\ell-1}$ whenever
$x>6+K6^6=6+6531840(k+14)$.
The threshold on $n$ and $\rho>\sqrt{2n}$ guarantee this inequality. Hence a core that loses in normalized moment order also loses at the actual spectral root, by~\eqref{eq:resolvent}.

The congruence makes $(N-v(C_{k,E}))/7$ an integer, and the lower bound on $n$ makes it nonnegative. The constructions in Section~\ref{sec:constructions} prove that the stated graph is 1-planar and $tC_5$-free. Equality in the local classification is graph isomorphism, so the extremizer is unique.
\end{proof}

\section{Computational verification and reproducibility}\label{sec:repro}
Lemma~\ref{lem:finitecomponents} is verified by exact finite enumeration. After Theorems~\ref{thm:cover} and~\ref{thm:defect} reduce the problem to defect at most six, the possible orders are finite and every candidate is represented by a fixed-budget edge-deletion set in an explicit envelope $A_m$. Three exhaustive implementations return the same deletion sets and exact moment values, and a separate five-subset/permutation enumerator verifies all pentagons in the relevant $A_m$ and $W_m$. The covering theorem, the packing--defect theorem, and the positive-packing classification are proved analytically.

The supplementary material records the exact search ranges, algorithms, certificates, independent checks, and reproduction commands.

\section{Concluding remarks}
Theorem~\ref{thm:main} resolves the fixed-$t$ disjoint-$C_5$ problem posed in~\cite[Problem~1]{LWZ}. The proof also yields two structural results that do not depend on the final residue-class classification: Theorem~\ref{thm:cover} restricts every connected 1-planar $K_2$-join remainder to two path-square envelopes, and Theorem~\ref{thm:defect} converts bounded pentagon packing into a sharp linear defect bound. The latter gives the exact edge extremum in Corollary~\ref{cor:edge}.

For the spectral problem, the zero-defect block $B=K_1\vee2K_3$ accounts for the unbounded part of an extremizer. The normalized resolvent removes these repeated blocks exactly, leaving only a bounded core. The strict moment inequalities in Section~\ref{sec:condensation} then force the packing and the nonzero defect into that single core. The quantitative threshold on $n$ is not optimized; it is chosen only to make the reduction and resolvent-tail estimates explicit. It would be natural to ask whether other fixed cycle packings in 1-planar graphs admit analogous sharp defect laws and finite-core reductions.

\section*{Data and code availability}
All code, exact certificates, and reproduction scripts used for Lemma~\ref{lem:finitecomponents} and the supplementary checks are supplied in a separate supplementary code archive. The archive includes a SHA-256 manifest and a one-command verification runner; the verification supplement records the exact search ranges and reproduction commands.

\section*{Declaration of generative AI and AI-assisted technologies in the manuscript preparation process}
ChatGPT (OpenAI) was used during manuscript preparation for language editing and limited technical assistance. The author reviewed and takes full responsibility for the final manuscript.

\end{document}